\documentclass[a4paper,12pt]{article}

\usepackage{amsmath, wrapfig}
\usepackage{dsfont, a4wide, amsthm, amssymb, amsfonts, graphicx}
\usepackage{fancyhdr, xspace, psfrag, setspace, supertabular, color, enumerate}
\usepackage{hyperref}

\newtheorem{theorem}{Theorem}[section]
\newtheorem{proposition}[theorem]{Proposition}
\newtheorem{lemma}[theorem]{Lemma}

\newtheorem{definition}[theorem]{Definition}

\newtheorem{example}[theorem]{Example}

\newtheorem*{example*}{Example}

\def\a{\alpha}

\def\d{\delta}
\def\e{\epsilon}

\def\ca{\mathcal{A}}
\def\cb{\mathcal{B}}
\def\cc{\mathcal{C}}
\def\cf{\mathcal{F}}

\def\F{\mathbb{F}}

\def\E{\mathbb{E}}
\def\P{\mathbb{P}}
\def\J{\mathbb{J}}
\def\HH{\mathbb{H}}

\DeclareMathOperator{\cor}{cor}
\DeclareMathOperator{\Diam}{Diam}
\DeclareMathOperator{\TV}{TV}
\DeclareMathOperator{\cod}{cod}

\title{On the expressive power of mod-$p$ linear forms on the Boolean cube}
\author{Thomas Karam\footnote{Mathematical Institute, University of Oxford. This research was essentially carried out when that author was at the Department of Pure Mathematics and Mathematical Sciences, University of Cambridge. Email: \texttt{thomas.karam@maths.ox.ac.uk}.}}
\begin{document}
\maketitle
%\section{}
%\subsection{}

\begin{abstract}

Let $(\ca_i)_{i \in [s]}$ be a sequence of dense subsets of the Boolean cube $\{0,1\}^n$ and let $p$ be a prime. We show that if $s$ is assumed to be superpolynomial in $n$ then we can find distinct $i,j$ such that the two distributions of every mod-$p$ linear form on $\ca_i$ and $\ca_j$ are almost positively correlated. We also prove that if $s$ is merely assumed to be sufficiently large independently of $n$ then we may require the two distributions to have overlap bounded below by a positive quantity depending on $p$ only. 
\end{abstract}

\tableofcontents

\section{Introduction}

For $p$ a prime and $n$ a positive integer, we will throughout say that a \emph{mod-$p$ linear form}, or \emph{mod-$p$ form} for short, is a linear form $\phi: \F_p^n \to \F_p$. Such a form can be written as \[\phi(x) = a_1 x_1 + \dots + a_n x_n\] for some $a_1, \dots, a_n \in \F_p$, and we will say that the \emph{support} of $\phi$ is the set \[\{z \in [n]: x_z \neq 0\}.\] If $\phi, \psi$ are two mod-$p$ forms then we will say that the \emph{support distance} between them is the size of the support of the difference $\phi - \psi$.

\subsection{Restrictions of mod-$p$ forms to the cube}

One central tool in the analysis of Boolean functions and in theoretical computer science more broadly is discrete Fourier analysis. There, the characters involved in the definition of the Fourier coefficients are functions $\F_2^n \to \F_2$ of the type $x \to (-1)^{\phi(x)}$, for some linear form $\phi: \{0,1\}^n \to \F_2$, where $\{0,1\}$ is identified with $\F_2$. Restrictions to $\{0,1\}^n$ of mod-$p$ forms for some general prime $p$ appear to be variants of these objects that are natural to study. Because $\{0,1\}$ can be identified to the whole of $\F_2$, whereas $\{0,1\}$ can only be embedded into $\F_p$, we can expect behaviour arising from linear forms with $p \ge 3$ on the cube to be at least a little more complex, and in some ways this is indeed the case. 

A basic observation which illustrates some of the difficulties that arise is the following. A given family $(\phi_1, \dots, \phi_k)$ of linear forms $\F_2^n \to \F_2$ is linearly independent if and only if it is independent in the probabilistic sense that when $x$ is an element of $\{0,1\}^n$ chosen uniformly at random, the $k$ events \[\phi_1(x) = y_1, \dots, \phi_k(x) = y_k\] are jointly independent, for any given $(y_1, \dots, y_k) \in \F_2^k$. On the other hand we do not have such an equivalence for restrictions to $\{0,1\}^n$ of mod-$p$ forms. For instance, the mod-$p$ forms $x_1$ and $x_1+x_2$ are linearly independent, but when $(x_1,x_2)$ is chosen uniformly at random in $\{0,1\}^2$, conditioning on the event $x_1 = 0$ decreases (from $1/4$ to $0$) the probability of the event $x_1+x_2 = 2$. Returning to the topic of Fourier analysis, likewise the character-like functions defined by $\omega_p^{\phi(x)}$ with $\phi$ some mod-$p$ form do not have their inner products of the type \[\E_{x \in \{0,1\}^n} \omega_p^{\phi_1(x)} \omega_p^{-\phi_2(x)} \] equal to $0$ in general, even if the mod-$p$ forms $\phi_1, \phi_2$ are linearly independent, although these inner products become close to $0$ if the difference $\phi_1 - \phi_2$ has large support.

It is nonetheless not the case that these restrictions to $\{0,1\}^n$ of mod-$p$ forms are so unstructured as to not lead themselves to basic statements analogous to those that would arise from linear forms $\F_2^n \to \F_2$, or more generally from unrestricted linear forms over finite fields. For instance, a special case of the fact above - the fact that \[\E_{x \in \{0,1\}^n} \omega_p^{\phi(x)}\] is close to $0$ if $\phi$ has large support - 
was used by Gowers and the author in \cite{Gowers and K approximation} to show that subsets of $\{0,1\}^n$ satisfying systems of conditions of the type \[\phi_1(x) \in E_1, \dots, \phi_k(x) \in E_k\] with $\phi_1, \dots, \phi_k$ some mod-$p$ forms and $E_1, \dots, E_k \subset \F_p$ can be arbitrarily well approximated from the inside by sets defined in a similar way but using only a bounded number of conditions.

In the setting of polynomials, it was established \cite{Gowers and K equidistribution} by the same authors that a low-degree polynomial $\F_p^n \to \F_p$ which is not approximately equidistributed on $\{0,1\}^n$ can be expressed (up to a polynomial vanishing on $\{0,1\}^n$) in a simple way in terms of a bounded number of polynomials with strictly smaller degree, a property that had previously been shown by Green and Tao \cite{Green and Tao} to be true when the assumption held for the distribution of this polynomial on the whole of $\F_p^n$, and then studied further both qualitatively and quantitatively in that unrestricted setting by various authors such as Lovett, Kaufman, Bhowmick, Janzer, Milicevic, Moshkovitz, Cohen, Zhu, Ziegler, Kazhdan, and Adiprasito. Later, it was also shown \cite{K} that under assumptions on the behaviour of the polynomial on $\{0,1\}^n$ involving its range rather than its distribution, we can furthermore require that the polynomial be expressed in a simple way in terms of polynomials of yet smaller degree.

What the three previous results have in common is that a property defined using linear forms or polynomials $\F_p^n \to \F_p$ was shown to also hold or to have a simple weakening after restriction to $\{0,1\}^n$, so it is reasonable to hope that this will be the case for more properties.

Besides the fact that some properties extend nicely to mod-$p$ forms, one concrete motivation for considering their behaviour on sets of the type $S^n$ for some strict subset $S$ of $\F_p$ are their applications to Ramsey theory, as several obstructions can be naturally formulated in terms of these mod-$p$ forms. For instance, one special case of the density Hales-Jewett theorem, a strengthening of the Hales-Jewett theorem \cite{Hales-Jewett} originally proved by Furstenberg and Katznelson \cite{Furstenberg and Katznelson k=3}, \cite{Furstenberg and Katznelson} and then proved again by the Polymath1 project \cite{Polymath}, states that for any fixed $\d>0$, there exists a positive integer $n$ such that every subset of $[3]^n$ with density at least $\d$ inside $[3]^n$ must contain a combinatorial line, that is, some triple $(x,y,z)$ of elements of $[3]^n$ such that for some partition $\{X_1,X_2,X_3,W\}$ of $[n]$ we have $x_i = y_i = z_i = j$ for every $j \in [3]$ and every $i \in X_j$, as well as $x_i = 0, y_i = 1, z_i = 2$ for every $i \in W$. If $A,B,C$ are subsets of $[3]^n$ such that $A \times B \times C$ contains a combinatorial line $(x,y,z)$, then for every mod-$p$ form $\phi$, linearity shows that the triple $(\phi(x), \phi(y), \phi(z))$ must be an arithmetic progression in $\F_p$. In other words, if for some mod-$p$ form $\phi$ the triple of images $(\phi(A), \phi(B), \phi(C))$ does not contain an arithmetic progression, then $A \times B \times C$ contains no combinatorial line. Similar obstructions defined using mod-$p$ forms, some of which involving $\{0,1\}^n$ rather than $\{0,1,2\}^n$, will be discussed and ruled out in the upcoming paper \cite{Gowers and K obstruction} in the context of a conjectured polynomial generalisation of the density Hales-Jewett conjecture and of its simplest unsolved case discussed in \cite{Gowers}.

In the present paper, rather than focus on obstructions we will use mod-$p$ forms to address a more basic kind of question which arises in numerous contexts as a way of measuring the expressive power of a set of functions or of data. Given a class $\cc$ of objects and a class $\cf$ of functions defined on $\cc$, how many objects in $\cc$ can we select such that for any two of them we can find a function in $\cf$ that clearly separates them ? In our case, the class $\cc$ of objects will be that of sufficiently dense subsets of the cube $\{0,1\}^n$ and the class $\cf$ of functions will be the functions sending a set in $\cc$ to the distribution of a mod-$p$ form on that set. Answering this question happens to furthermore be of independent Ramsey-theoretic interest, since it can be equivalently formulated as follows: how many dense subsets of the cube can we choose until we necessarily can find a pair of them for which the pair of distributions of every mod-$p$ form are not too far apart ?

\subsection{Notions of separation on the distributions of mod-$p$ forms}

We begin by defining a number of ways to measure proximity between two probability distributions.

\begin{definition}\label{Ways to measure proximity between two probability distributions}
	
	Let $m \ge 2$ be an integer. For any $m$-tuple $(X_1, \dots, X_m)$ of distributions each taking values in the same finite set $D$, we define the following quantities. 
\begin{enumerate}[(i)]
\item The \emph{diameter} $\Diam(X_1, \dots, X_m)$ is defined to be $\max\limits_{1 \le i < j \le m} \TV(X_i, X_j)$, where \[\TV(X_i, X_j)= \sum\limits_{d \in D} |\P(X_i=d) - \P(X_j=d)|\] is the total variation distance between $X_i$ and $X_j$ for all $i,j \in [m]$. 
\item The \emph{correlation} $\cor(X_1, \dots, X_m)$ is defined to be 
\[\big(\sum\limits_{d \in D} \P[X_1 = d] \dots \P[X_m = d]\big) - 1/|D|^{m-1}.\] 
\item The \emph{overlap} $\omega(X_1, \dots, X_m)$ is defined to be \[\sum\limits_{d \in D} \min(\P[X_1 = d], \dots, \P[X_m = d]).\]  \end{enumerate}
\end{definition}

We note in particular that if all but at most one of the variables $X_1, \dots, X_m$ are uniformly distributed (resp. approximately uniformly distributed), then $\cor(X_1, \dots, X_m)$ is zero (resp. small in absolute value), and that if $\omega(X_1, \dots, X_m)>0$, then the intersection of the ranges of $X_1, \dots, X_m$ is not empty.

These definitions allow us to establish a qualitative hierarchy in the extent to which distributions of random variables $X_1, \dots, X_m$ are not too far apart. Our definition for the correlation may be slightly surprising, but the centered expression \[\sum_{d \in D} (\P[X_1 = d] - |D|^{-1}) \dots (\P[X_m = d] - |D|^{-1}),\] which might appear at first to be more natural to consider, does not quantify the extent to which the $m$ distributions are similar: if for instance $m=3$, $D=\F_3$ and $X_1, X_2, X_3$ satisfy $\P[X_i = d] = \delta_{i,d}$ for all $i,d \in [3]$ then that quantity is equal to $2/9$, which is positive, even though the ranges of $X_1,X_2,X_3$ are pairwise disjoint. Each of the following properties is qualitatively stronger than the next, in the sense that whenever $i \in \{1,2,3\}$, for every choice of the parameters in the $i+1$th property which is sufficiently close to $0$ (in a manner that depends on $|D|$ and $m$), there is a choice of parameters (which may further depend on $|D|$ and $m$) in the $i$th property which implies it. \\

\begin{enumerate}
\item Close distributions for some $\eta >0$:  \[\Diam(X_1, \dots X_m) \le \eta.\]
\item Almost positive correlation for some $\nu>0$: \[\cor(X_1, \dots X_m) \ge - \nu.\]
\item Overlap bounded away from $0$ for some $A>0$: \[\omega(X_1, \dots X_m) \ge A.\]
\item Overlapping distributions: \[\omega(X_1, \dots X_m) > 0.\]
\end{enumerate}

\begin{proof} The first implication follows from H\"older's inequality: if $\Diam(X_1, \dots X_m) \le \eta$, then $\TV(X_1, X_i) \le \eta$ for each $i \in \lbrack m \rbrack$, so \[\sum_{d \in D} \P[X_1=d] \dots \P[X_m=d] \ge \sum_{d \in D} (\P[X_1=d]-\eta)^m \ge |D|^{-(m-1)} (1-|D| \eta)\] and hence \[\cor(X_1, \dots X_m) \ge - |D|^{m-2} \eta.\] The second implication follows from a short calculation: if $\cor(X_1, \dots X_m) \ge - \nu$ then \[\sum_{d \in D} \P[X_1 = d] \dots \P[X_m = d] \ge 1/|D|^{m-1} - \nu,\] so there exists $d \in D$ satisfying \[\P[X_1 = d] \dots \P[X_m = d] \ge 1/|D|^{m} - \nu/|D|,\] and hence \[\P[X_i = d] \ge 1/|D|^{m} - \nu/|D|\] for each $i \in \lbrack m \rbrack$, from which \[\omega(X_1, \dots, X_m) \ge 1/|D|^{m} - \nu/|D|\] follows. The third implication is immediate. \end{proof}

If $E$ is an event and $\ca$ is a non-empty subset of $\{0,1\}^n$ then we shall write $\P_{\ca}[E]$ for the probability that $x$ satisfies $E$ when $x$ is chosen uniformly at random inside $\ca$. If $F$ is a function defined on $\{0,1\}^n$ and $\ca, \cb$ are non-empty subsets of $\{0,1\}^n$, then we write $\TV_{\ca, \cb}(F)$ for the total variation distance between the distributions of $F(x)$ when $x$ is chosen uniformly at random from $\ca$ and when it is chosen uniformly at random from $\cb$. If $m \ge 2$ is a positive integer and $\ca_1, \dots, \ca_m$ are non-empty subsets of $\{0,1\}^n$, then we write $\Diam_{\ca_1, \dots, \ca_m}(F)$ for $\Diam(X_1, \dots, X_m)$, where for each $i \in \lbrack m \rbrack$ the distribution of the variable $X_i$ is the distribution of $F(x)$ for $x$ chosen uniformly at random in $\ca_i$. We define $\cor_{\ca_1, \dots, \ca_m}(F)$ and $\omega_{\ca_1, \dots, \ca_m}(F)$ in a similar way.

\subsection{Main results}

Starting with a number $s$ of non-empty subsets $\ca_1, \dots, \ca_s$ of $\{0,1\}^n$, we can ask whether it is possible to obtain a pair $(i,j)$ of distinct elements of $\lbrack s \rbrack$ such that all mod-$p$ forms have close distributions on the pair $(\ca_i, \ca_j)$. As we will now illustrate, this is not possible in general, even if the number $s$ grows exponentially with $n$.

In the case $p=2$, that can be seen from the fact that the $2^{n}-1$ non-zero linear forms $\phi_1, \dots, \phi_{2^n-1}: \F_2^n \to \F_2$ each have a different linear hyperplane of $\F_2^n$ as their kernels; since these hyperplanes have pairwise intersections of size $2^{n-2}$, the hyperplanes $\ca_i = \ker \phi_i$ with $i \in [2^n-1]$ are such that $\phi_i(\ca_i)$ only takes the value $0$ on  $\ca_i$ but takes the values $0$ and $1$ with equal probability on every $\ca_j$ with $j \neq i$.

For $p \ge 3$, it is no longer true in general that two different mod-$p$ forms have the same preimage of $0$ inside $\{0,1\}^n$, even if they are proportional: for instance the mod-$5$ forms $x_1 + x_2$ and $2x_1 + x_2$ both have the same preimage $\{x_1=x_2=0\}$ of $0$ inside $\{0,1\}^n$. However, we can find a set of mod-$p$ forms with size exponential in $n$ and within which this is the case. We begin by partitioning $[n]$ into pairwise disjoint sets $I_1,\dots,I_{2d}$ each of size $p$, with $d = \lfloor n/2p \rfloor$ and a remainder set $I_0$ with size at most $2p-1$. We then restrict our attention to the mod-$p$ forms $\phi$ of the type \[\phi(x) = a_1 x_1 + \dots + a_n x_n\] where for every $u \in [d]$, all coefficients $a_z$ with $z \in I_{2u-1} \cup I_{2u}$ are the same. We can then find a maximal set of such mod-$p$ forms with size at least $p^{d-1} = p^{\lfloor n/2p \rfloor-1}$ such that no two of them are proportional. We then index the forms in this set as $\phi_1, \dots, \phi_t$.

We then take $x$ to be the element of $\{0,1\}^n$ such that $x_i = 0$ whenever $i \in I_0$ or i in $I_u$ for some odd $u$, and $x_i = 1$ whenever $i \in I_u$ for some non-zero even $u$. We have $\phi_i(x) = 0$ for every $i \in [t]$, since the contribution of each set $I_u$ of coordinates is zero. Let now $i,j$ be distinct indices in $[t]$. Because $\phi_i$ and $\phi_{j}$ are not proportional, there exist distinct $u_1, u_2 \in [d]$ such that the pair $(A_{1,i}, A_{2,i})$ of coefficients of $\phi_i$ on $I_{2u_1-1} \cup I_{2u_1}$ and on $I_{2u_2-1} \cup I_{2u_2}$ is not proportional to the pair $(A_{1,j}, A_{2,j})$ of coefficients of $\phi_j$ on $I_{2u_1-1} \cup I_{2u_1}$ and on $I_{2u_2-1} \cup I_{2u_2}$. Viewing $A_{1,i}, A_{2,i}, A_{1,j}, A_{2,j}$ as integers in $[0,p-1]$, we use $x$ to define a new element $y$ of $\{0,1\}^n$ by changing $A_{2,i}$ of the coordinates of $x$ in $I_{2u_1-1}$ to $1$ and changing $A_{1,i}$ of the coordinates of $x$ in $I_{2u_2}$ to $0$. We then have \begin{align*} \phi_i(y) - \phi_i(x) &= A_{2,i} A_{1,i} - A_{1,i} A_{2,i} = 0\\
\phi_j(y) - \phi_j(x) &= A_{2,i} A_{1,j} - A_{1,i} A_{2,j} \neq 0, \end{align*} so $\phi_i(y) = 0$ but $\phi_j(y) \neq 0$. We have shown that for every pair $(i,j)$ of distinct elements of $[t]$, the set $\phi_i^{-1}(0)$ is not contained in the set $\phi_j^{-1}(0)$.

One of the tools that we will introduce later, Proposition \ref{Lower bound on non-zero probabilities}, shows that any event on $\{0,1\}^n$ defined by a bounded number $k$ of mod-$p$ forms has probability either equal to 0 or bounded below in a way that depends on $p$ and $k$ only. This implies in particular a positive lower bound that depends only on $p$ on the probability, for $x$ chosen uniformly in $\{0,1\}^n$, of the event \[\phi_j(x) = 0, \phi_i(x) \neq 0,\] and this now allows us to conclude as we did for $p=2$.

\begin{example}\label{No close distributions} Let $p \ge 2$ be a prime. The subsets $\ca_1 = \phi_1^{-1}(0), \dots, \ca_t = \phi_t^{-1}(0)$ of $\{0,1\}^n$ each have density at least $2^{-(p-1)}$ inside $\{0,1\}^n$ and satisfy \[\TV_{\ca_i,\ca_j} \phi_i \ge \P_{\ca_i} [\phi_i = 0] - \P_{\ca_j} [\phi_i = 0] \ge 2^{-2(p-1)}\] for any pair $(i,j)$ of distinct elements of $[t]$.

\begin{proof} The lower bound on the densities of the subsets $\ca_i$ follows from Proposition \ref{Lower bound on non-zero probabilities}. The first inequality follows from the definition. The second follows from the calculation \begin{align*} \P_{\ca_i} [\phi_i = 0] - \P_{\ca_j} [\phi_i = 0] & = 1 - \P[\phi_i = 0, \phi_j = 0]/\P[\phi_j = 0] \\ & = 1 - (\P[\phi_j = 0] - \P[\phi_j = 0, \phi_i \neq 0]) / \P[\phi_j = 0] \\ & \ge 1 - (\P[\phi_j = 0] - 2^{-2(p-1)}) / \P[\phi_j = 0] \\ & \ge 1 - (1-2^{-2(p-1)}) \end{align*} where all probabilities on the right-hand side are taken with respect to $x$ chosen uniformly in $\{0,1\}^n$, and where in the third line we apply Proposition \ref{Lower bound on non-zero probabilities}.\end{proof} \end{example}

We will nonetheless begin by showing that close distributions can be obtained if we only consider mod-$p$ forms with bounded support size and assume the number of dense subsets of $\{0,1\}^n$ to grow superpolynomially in $n$.

\begin{proposition}\label{Close distributions for all functions depending on a bounded number of coordinates} Let $\delta>0$, let $k \ge 1$ be an integer, and let $\nu>0$. Then there exists $c_{\TV}(\d, \nu, k)>0$ such that that the following holds. If $s$ is a positive integer and $\ca_1, \dots, \ca_s$ are non-empty subsets of $\{0,1\}^n$ with density at least $\delta$ inside $\{0,1\}^n$, there exists a positive integer $N(\nu, k)$ such that for all $n \ge N(\nu, k)$, there exists a subset $S \subset \lbrack s \rbrack$ with size at least $s/(n+1)^{c_{\TV}(\d, \nu, k)}$ such that \[\Diam_{\ca_i: i \in S} g \le \nu \] for every function $g$ defined on $\{0,1\}^n$ and determined by at most $k$ coordinates.
\end{proposition}

Proposition \ref{Close distributions for all functions depending on a bounded number of coordinates} does not involve the structure of the function $g$, and in particular does not require $g$ to be a mod-$p$ form. However, we will use some of the lemmas and ideas involved in the proof of Proposition \ref{Close distributions for all functions depending on a bounded number of coordinates} in our later proofs, and in the case where $g$ is a mod-$p$ form it seems worthwhile to compare its conclusion to those that we can or cannot obtain when $g$ is not assumed to have bounded support size, as we will do in the final summarising table of the present section.

We note that the bound that we get in Proposition \ref{Close distributions for all functions depending on a bounded number of coordinates} cannot be replaced by a sufficiently low power of $n$, even for a fixed value of $k$ and even if we are only aiming for a conclusion on all mod-$p$ forms with bounded support rather than on more general functions.

\begin{example}

Let $k \ge 1$ be an integer. For every subset $I$ of $[n]$ with size $k$ let $\ca_I$ be the set of $x \in \{0,1\}^n$ such that $x_i = 0$ for every $i \in I$, and let $\phi_I$ be the mod-$p$ form such that $\phi(x)$ is defined to be the sum of the coordinates $x_i$ with $i \in I$. Then the sets $\ca_I$ each have density at least $2^{-(p-1)}$, and whenever $I,J$ are distinct subsets of [n] with size $k$ we have \[\P_{\ca_I}(\phi_I = 0) - \P_{\ca_J}(\phi_I = 0) \ge 1/2.\]

\end{example}

\begin{proof} The lower bound on the densities of the sets $\ca_I$ comes from Proposition \ref{Lower bound on non-zero probabilities}. The inequality follows from pairing elements of $\{0,1\}^n$ according to one of the coordinates in $J \setminus I$. For any fixed choice of any of the other coordinates, either zero or both elements of the pair belong to $\ca_I$ but at most one element of each pair belongs to $\ca_J$. \end{proof}

Let us now examine what we can or cannot hope for in that case, where we ask for a result on all mod-$p$ forms rather than merely those with bounded support size. We may first aim for a weaker conclusion and ask for a pair $(i,j)$ of distinct elements of $\lbrack s \rbrack$ such that all mod-$p$ forms have almost positive correlation on the pair $(\ca_i, \ca_j)$. If the number $s$ of sets $\ca_i$ is only known to tend to $\infty$ with $n$ then the following example shows that for $p \ge 3$ it is still not possible in general to obtain such a pair. To state it, it will be convenient to identify elements $x \in \{0,1\}^n$ with respective subsets $A \subset [n]$ defined by $A = \{i \in [n]: x_i = 1\}$. We note that this example does not apply to $p=2$, as it then leads to a correlation of zero, rather than to a negative correlation.

\begin{example} Let $p \ge 3$ be a prime, let $s= \lfloor n/p \rfloor$, let $Z_1, \dots, Z_s$ be pairwise disjoint subsets of $[n]$ each with size $\frac{p-1}{2}$, and let \[\ca_i:= \{ A \subset [n]: |A \cap Z_i| = 0\}\] for each $i \in \lbrack s \rbrack$. Then for each pair $(i,j)$ of distinct elements of $\lbrack s \rbrack$ the correlation $\cor_{\ca_i,\ca_j} \phi_{i,j}$ of the form \[\phi_{i,j}: A \mapsto |A \cap Z_i| - |A \cap Z_j|\] is equal to $2^{-(p-1)} - 1/p$, which is negative. \end{example}

However, we shall show that the statement that we have just been aiming for becomes true if we modify it in either of two ways. In one direction, if $s$ is superpolynomial in $n$ then it is always possible to find the desired pair $(i,j)$ such that every mod-$p$ form has almost positive correlation on the pair $(\ca_i, \ca_j)$. In another direction, merely having $s$ tend to infinity with $n$ is enough to ensure the existence of such a pair $(i,j)$ of distinct elements of $\lbrack s \rbrack$ such that every mod-$p$ form has overlap bounded below on the pair $(\ca_i, \ca_j)$ by some function of $p$ only. The following two theorems will be the main results of this paper.

\begin{theorem}\label{Almost positive correlations}

Let $p$ be a prime, let $m \ge 2$ be an integer, and let $\delta>0, \nu>0, \epsilon>0$. Then there exists $C_{\cor}(\delta, \nu, \epsilon,p,m) > 0$ such that the following holds. If $s$ is a positive integer and $\ca_1, \dots, \ca_s$ are non-empty subsets of $\{0,1\}^n$ with density at least $\delta$ inside $\{0,1\}^n$, then there exists a subset $R$ of $\lbrack s \rbrack$ with size at least $s/(n+1)^{C_{\cor}(\delta, \nu, \epsilon,p,m)}$ such that for at least $(1-\epsilon)|R|^m$ of $m$-tuples $(i_1, \dots, i_m) \in R^m$ we have \[\cor_{\ca_{i_1}, \dots, \ca_{i_m}} \phi \ge - \nu\] for every mod-$p$ form.

\end{theorem}

\begin{theorem}\label{Overlap bounded below}

Let $p$ be a prime, let $m \ge 2$ be an integer, and let $\delta>0, \nu>0, \epsilon>0$. Then there exists $c_{\omega}(\delta, \nu, \epsilon,p,m) > 0$ such that the following holds. If $s$ is a positive integer and $\ca_1, \dots, \ca_s$ are non-empty subsets of $\{0,1\}^n$ with density at least $\delta$ inside $\{0,1\}^n$, then there exists a subset $Q$ of $\lbrack s \rbrack$ with size at least $c_{\omega}(\delta, \nu, \epsilon,p,m)s$ such that for at least $(1-\epsilon)|Q|^m$ of the $m$-tuples $(i_1, \dots, i_m) \in Q^m$ we have \[\omega_{\ca_{i_1}, \dots, \ca_{i_m}} \phi \ge 2^{-(m+1)(p-1)} - \nu \] for every mod-$p$ form.

\end{theorem}

Our results and counterexamples allow us to obtain the following table, which summarises for any prime $p \ge 3$ and any integer $m \ge 2$ whether from a collection of dense subsets of the cube we can always find an $m$-tuple of subsets guaranteeing the properties we have been discussing on all mod-$p$ forms simultaneously, depending on how many subsets we begin with.

\begin{table}[h]
	\begin{tabular}{p{4.5cm}|l|l}
		Number of subsets $\ca_i$ & Tending to $\infty$ & $\Omega(n^C)$ for all $C>0$  \\
		\hline
		Overlapping distributions & Yes & Yes \\
		\hline
		Overlap bounded below away from zero & Yes & Yes\\
		\hline
		Close distributions when the support has bounded size & No & Yes\\
		\hline
		Almost positive correlation & No & Yes\\
		\hline
		Close distributions & No & No
		
	\end{tabular}
\end{table}

We organise the remainder of the paper in two sections. In Section \ref{Section: Tools involved in the proofs} we will first recall and build various technical tools about the distributions of mod-$p$ linear forms as well as some information-theoretic tools. Then, Section \ref{Section: Proofs of the main theorems} will be devoted to the proofs of our main results, primarily Theorem \ref{Almost positive correlations} and Theorem \ref{Overlap bounded below}.

\section*{Acknowledgement}

The author thanks Timothy Gowers for introducing him to the set version of Lemma \ref{Common element lemma} at a very early stage in the process that ultimately led to the present paper.

\section{Tools involved in the proofs} \label{Section: Tools involved in the proofs}

\subsection{Basic results on mod-$p$ linear forms on the cube}

We begin this section by recalling properties on the distribution of mod-$p$ linear forms on $\{0,1\}^n$. Throughout this paper we will mainly use two facts regarding them, which are special cases of respectively \cite{Gowers and K approximation}, Proposition 2.4 and \cite{K}, Lemma 2.3. The first property informally states that a set of several mod-$p$ forms that is well-separated is approximately jointly equidistributed on $\{0,1\}^n$.

\begin{proposition} \label{Equidistribution of k-tuples of linear forms}

Let $p$ be a prime, and let $k \ge 1$, $r \ge 0$ be integers. If $\phi_1, \dots, \phi_k: \F_p^n \to \F_p$ are linear forms such that the support size of the linear combination $a_1 \phi_1 + \dots + a_k \phi_k$ is at least $r$ for every $(a_1, \dots, a_k) \in \F_p^k \setminus \{0\}$, then \[|\P_{x \in \{0,1\}^n}[\phi_1(x) = y_1, \dots, \phi_k(x) = y_k] - p^{-k}| \le (1-p^{-2})^r\] for every $(y_1, \dots, y_k) \in \F_p^k$.

\end{proposition}

The second property states that the probability that a $k$-tuple of mod-$p$ forms takes a given $k$-tuple of values is always either zero or bounded below by a quantity that depends only on $p$ and $k$. This fact is not necessary to obtain Theorem \ref{Almost positive correlations} but will play an important role in the proof of Theorem \ref{Overlap bounded below}.

\begin{proposition} \label{Lower bound on non-zero probabilities} Let $p$ be a prime, let $k \ge 1$ be an integer, and let $\phi_1, \dots, \phi_k: \F_p^n \to \F_p$ be linear forms. Then for any $(y_1, \dots, y_k) \in \F_p^k$ the probability \[\P_{x \in \{0,1\}^n}[\phi_1(x) = y_1, \dots, \phi_k(x) = y_k]\] is either $0$ or at least $2^{-k(p-1)}$. \end{proposition}

\subsection{Information-theoretic tools}

In this subsection we recall tools from information theory and prove some statements that will be repeatedly useful to us throughout this paper.

The \emph{Shannon entropy} $\HH(X)$ of a random variable $X$ taking values in a finite set $D$ is defined by the formula \[ \HH(X) = - \sum_{x \in D} \P[X=x] \log \P[X=x].\]

It is a standard inequality that $\HH(X) \le \log |D|$ and that $\HH(X) = \log |D|$ if and only if $X$ is uniformly distributed on $D$. We will quantify how far $X$ is from the uniform distribution by the difference \[\mathbb{J}(X)= \log |D| - \mathbb{H}(X),\] which is often called the \emph{negentropy} of $X$, between the largest possible entropy $\log |D|$ and the entropy of $X$. An expansion around the uniform distribution provides the following estimate which will be convenient for us to use in our arguments, although a more qualitative statement would suffice for them to work.

\begin{lemma}\label{Low negentropy implies approximate equidistribution} Let $X$ be a random variable taking values in a finite set $D$. Then there exists $\eta_0(D)>0$ such that for every $\eta \in \lbrack 0, \eta_0)$, if $\J(X) \le \eta$ then \[\sum_{x \in D} (\P[X=x] - 1/|D|)^2 \le (8/|D|) \J(X)\] and in particular \[|\P[X=x] - 1/|D|| \le (8\eta/|D|)^{1/2}\] for every $x \in D$. \end{lemma}

For $X,Y$ random variables taking values respectively in finite domains $D_X$ and $D_Y$, the conditional entropy defined by \[\HH(X|Y) = \mathbb{H}((X,Y)) - \mathbb{H}(Y)\] satisfies \[\HH(X|Y) \le \HH(X) \text{ and } \HH((X,Y)) \le \HH(X) + \HH(Y).\] We similarly define \[\mathbb{J}(X|Y) = \mathbb{J}((X,Y)) - \mathbb{J}(Y)\] and the inequalities \[\mathbb{J}(X|Y) \ge \mathbb{J}(X) \text{ and } \mathbb{J}((X,Y)) \ge \mathbb{J}(X) + \mathbb{J}(Y)\] follow from the corresponding reverse inequalities for the entropy. Just as the conditional entropy $\mathbb{H}(X|Y)$ can be expressed as \[ - \sum\limits_{y \in D_Y} \P[Y=y] \sum\limits_{x \in D_X} (\P[X=x| Y=y] \log \P[X=x | Y=y] ) = \mathbb{E}_{y \sim Y} \mathbb{H}(X|Y=y),\] the quantity $\mathbb{J}(X|Y)$ can in turn be written and interpreted as $\mathbb{E}_{y \sim Y} \mathbb{J}(X|Y=y)$. Therefore, one technical aspect to be aware of while using the quantity $\mathbb{J}(X|Y)$, and which will come up in our proofs, is that $\mathbb{J}(X|Y)$ being low does not by itself guarantee that for every $y \in D_Y$, $X$ is close to uniformly distributed conditionally on the event $Y=y$; this is however necessarily the case if the probability of this event is known to be bounded away from zero, since \begin{equation} \mathbb{J}(X) \ge \P[Y=y] \mathbb{J}(X|Y=y). \label{Markov inequality} \end{equation}

Whenever $F: \{0,1\}^n \rightarrow D$ is a function with finite codomain $D$ and $\mathcal{A}$ is a non-empty subset of $\{0,1\}^n$, we shall write $\HH_{\mathcal{A}}(F)$ and $\J_{\mathcal{A}}(F)$ respectively for the entropy and negentropy of the variable $F(x)$, where $x$ is chosen uniformly at random inside $\mathcal{A}$.

Throughout we will consider families of functions which on the whole of $\{0,1\}^n$ are either uniformly distributed (in the case of coordinate forms of the type $x_z$ for some $z \in \lbrack n \rbrack$, which have codomain $\{0,1\}$) or approximately uniformly distributed (in the case of mod-$p$ forms with high support) on their respective codomains $\{0,1\}$ and $\F_p$. As these functions have zero or low negentropy on the whole of $\{0,1\}^n$, the negentropy of these functions on a non-empty subset $\ca$ of $\{0,1\}^n$ is a useful measure of how different their distribution on $\ca$ is from their distribution on $\{0,1\}^n$.

Let us consider a family $\mathcal{F}$ of functions defined on $\{0,1\}^n$. We can for instance take the family \[\mathcal{F}_0:= \lbrace x_z: z \in \lbrack n \rbrack \rbrace\] of coordinate forms $\{0,1\}^n \rightarrow \{0,1\}$, which will be of particular relevance to us. We emphasize that the codomain of these coordinate forms will always be taken to be $\{0,1\}$, even when this is not explicitly mentioned.

For $\ca$ a non-empty subset of $\{0,1\}^n$, a first task is to attribute to $\ca$ a family $\mathcal{I}$ of sets of elements of $\lbrack n \rbrack$ which ``represents" the anomalous behaviour of $\ca$. For instance, if $\mathcal{F} = \mathcal{F}_0$ and \[ \ca = \lbrace A \in \{0,1\}^n: A(x_1) + A(x_2)= 1 \rbrace\] then we will want in particular \begin{enumerate}[(i)] \item $\lbrace 1 \rbrace$ not to belong to $\mathcal{I}$: on its own, the coordinate is uniformly distributed.
\item $\lbrace 1,2 \rbrace$ to belong to $\mathcal{I}$: the pair $(x_1, x_2)$ is far from uniformly distributed on $\ca$, even if $x_1$ and $x_2$ individually are.
\item $\lbrace 1,2,3 \rbrace$ not to belong to $\mathcal{I}$: while it is true that the distribution of $(x_1, x_2, x_3)$ is far from uniform, the coordinate $x_3$ contributes to this lack of uniformity neither on its own nor through its interactions with the coordinates $x_1$ and $x_2$. \end{enumerate}

We now give the formal definition that we will use.

\begin{definition}\label{Irreducible family}
	
Let $\ca$ be a non-empty subset of $\{0,1\}^n$, let $\mathcal{F}$ be a family of functions defined on $\{0,1\}^n$ and let $\eta>0$. For each positive integer $k \ge 1$ let $\mathcal{I}_{\le k}(\ca, \mathcal{F}, \eta)$ be the family of non-empty subsets $T$ of $\mathcal{F}$ with size at most $k$ that satisfy one of the following two conditions. \begin{enumerate}[(i)]
	
\item $|T| = \lbrace f \rbrace$ for some $f \in \mathcal{F}$ and $\J(f) \ge \eta$.
	
\item $|T|>1$ and furthermore \begin{equation} \J(T) - \sum\limits_{1 \le i \le t} \J(T_i) \ge \eta. \label{partition inequality} \end{equation} for every $t \ge 2$ and every partition $T = \bigcup_{1 \le i \le t} T_i$ of $T$ (into non-empty sets). \end{enumerate}

Whenever a non-empty subset $T$ of $\mathcal{F}$ with size $2 \le |T| \le k$ and a partition $\bigcup_{1 \le i \le t} T_i$ of $T$ with $t \ge 2$ do not satisfy \eqref{partition inequality}, we say that $T$ \emph{reduces to} $\bigcup_{1 \le i \le t} T_i$. If $f \in \mathcal{F}$ satisfies $\J(f) < \eta$, then we also say that $\lbrace f \rbrace$ \emph{reduces} to the empty set. For each positive integer $k \ge 1$, we define \[I_{\le k}(\ca, \mathcal{F}, \eta) = \bigcup_{T \in \mathcal{I}_{\le k} (\ca, \mathcal{F}, \eta)} T.\] We call a set \emph{irreducible} if it has size at most $k$ and belongs to $\mathcal{I}_{\le k}(\ca, \mathcal{F}, \eta)$, and \emph{reducible} if it is a subset of $\cf$ of size at most $k$ that is not irreducible. \end{definition}

It follows in particular from Definition \ref{Irreducible family} that if $k \ge 1$ is a positive integer, $F_1$ and $F_2$ are disjoint and non-empty subsets of $\mathcal{F}$ with $|F_1| + |F_2| \le k$, and $F_1 \cup F_2 \in \mathcal{I}_{\le k}(\ca, \mathcal{F}, \eta)$, then $\mathbb{J}(F_2 | F_1)$, which is equal to $\mathbb{J}(F_1 \cup F_2) - \mathbb{J}(F_1)$, is at least $\mathbb{J}(F_2) + \eta$, and hence is in particular at least $\eta$.

Also, if a set $T \subset \mathcal{F}$ with $|T| \le k$ has empty intersection with $I_{\le k}(\ca, \mathcal{F}, \eta)$, then it follows from an immediate induction on $|T|$ and the (negation of the) definition of an irreducible set that \[\mathbb{J}(T) \le (2|T|-1) \eta \le (2k-1) \eta.\] In particular, the joint distribution of $(f : f \in T)$ is therefore approximately uniform on the codomain of $(f : f \in T)$. We now show that this inequality holds more generally after conditioning on some subset $U$ of $\mathcal{F}$ of small size.

\begin{lemma}\label{Low conditional negentropy} Let $T, U$ be two disjoint subsets of $\mathcal{F}$ such that $1 \le |T| + |U| \le k$ and $T \cap I_{\le k}(\ca, \mathcal{F}, \eta)$ is empty. Then \[\mathbb{J}(T|U) \le (2k-1) \eta.\] \end{lemma}

\begin{proof}
	
By definition, $I_{\le k}(\ca, \mathcal{F}, \eta)$ contains all irreducible sets, so any set that is not contained in $I_{\le k}(\ca, \mathcal{F}, \eta)$ is reducible. We select a family of sets as follows. We first select the family $\lbrace \lbrace T \cup U \rbrace \rbrace$, and then iterate the following inductive step: if a set $F$ in the family reduces to a family $F_1, \dots, F_t$, then we deselect $F$ and select all sets $F_1, \dots, F_t$, stopping the iterations when we obtain a family where we can no longer reduce any set (including singletons to the empty set).
	
At the end of the process we obtain a partition $T \cup U = \bigcup_{i} A_i$ where the sets $A_i$ are all contained in \[(T \cup U) \cap I_{\le k}(\ca, \mathcal{F}, \eta) = U \cap I_{\le k}(\ca, \mathcal{F}, \eta).\] We have reduced at most $2|T \cup U|-1$ times, and at each reduction the sum of the negentropies of the selected sets of functions has decreased by at most $\eta$, so \[\mathbb{J}(T \cup U) \le \sum\limits_{i} \J(A_i) + (2|T \cup U|-1) \eta.\] Using furthermore that \[\sum\limits_{i} \mathbb{J}(A_i) \le \mathbb{J}(U \cap I_{\le k}(\ca, \mathcal{F}, \eta)) \le \J(U),\] and writing $\mathbb{J}(T|U) = \mathbb{J}(T \cup U) - \mathbb{J}(U)$ we conclude that \[\mathbb{J}(T|U) \le (2|T \cup U|-1) \eta.\] The result follows by the assumption on $|T| + |U|$. \end{proof}

We now show an upper bound on the size of $\mathcal{I}_{\le k}$ that depends on how far $\mathcal{F}$ is from a jointly uniform distribution on $\{0,1\}^n$.

\begin{proposition}\label{Bounded number of irreducible sets} Let $D$ be a finite set, let $f: \{0,1\}^n \rightarrow D$ be a function, and let $\mathcal{F}= \mathcal{F}_0 \cup \lbrace f \rbrace$. Let $\delta, \eta>0$. Let $\ca \subset \{0,1\}^n$ be a subset of density at least $\d$. Then the following facts hold. \begin{enumerate}[(i)]

\item The negentropy $J_{\ca} (\mathcal{F})$ satisfies the upper bound $J_{\ca} (\mathcal{F}) \le \log \delta^{-1} + \log |D|$.
	
\item The size $K= |\mathcal{I}_{\le k}(\ca, \mathcal{F}, \eta)|$ satisfies the inequality $J_{\ca} (\mathcal{F}) \ge \eta ((K/k)/k!)^{1/k}$.
	
\item The size $K$ satisfies the upper bound $K \le (k+1)! (\eta^{-1}\log (\delta^{-1}|D|))^k$. \end{enumerate}

\end{proposition}

\begin{proof}
	
We first prove (i). We start with the inequality \[\HH_{\ca}(x_1, \dots, x_n,f) \ge \HH_{\ca}(x_1, \dots, x_n),\] which implies that \begin{equation} \J_{\ca}(x_1, \dots, x_n, f) \le \J_{\ca}(x_1, \dots, x_n) + \log |D|. \label{first bound in proof of bounded number of irreducible sets} \end{equation} Since $\ca$ has density at least $\d$, we have \[\P_{\ca}[(x_1, \dots, x_n) = u] \le \delta^{-1} \P_{\{0,1\}^n}[(x_1, \dots, x_n) = u]\] for every $u \in \{0,1\}^n$, and therefore
\begin{align*} \HH_{\ca}(x_1, \dots, x_n) & \ge - \sum_{u \in \{0,1\}^n} \P_{\ca}[(x_1, \dots, x_n) = u] \log \P_{\{0,1\}^n}[(x_1, \dots, x_n) = u] - \log \delta^{-1}\\
& = \HH_{\{0,1\}^n}(x_1, \dots, x_n) - \log \delta^{-1},\end{align*} since $\P_{\{0,1\}^n}[(x_1, \dots, x_n) = u]$ is the same for all $u \in \{0,1\}^n$. We deduce the upper bound \[\J_{\ca}(x_1, \dots, x_n) \le \J_{\{0,1\}^n}(x_1, \dots, x_n) + \log \delta^{-1} = \log \delta^{-1}.\] Combining this with the inequality \eqref{first bound in proof of bounded number of irreducible sets} finishes the proof of (i).

We next prove (ii). By the pigeonhole principle, at least $K/k$ of the sets have common size $l \le k$. By the Erd\H{o}s-Rado sunflower theorem (proved in \cite{Erdos and Rado}) one can extract from them a sunflower of $r$ such sets as long as \[k!(r-1)^k \le K/k,\] or equivalently \[r \le 1+ ((K/k)/k!)^{1/k}.\] Writing the centre as $T_0$ and the petals as $T_1, \dots, T_r$ we have the lower bound \[ \mathbb{J}(\bigcup_{0 \le i \le r} T_i) \ge \mathbb{J}(\bigcup_{1 \le i \le r} T_i | T_0) \ge \sum\limits_{1 \le i \le r} \mathbb{J}(T_i | T_0) \ge r \eta \] which proves (ii).  It is then immediate to deduce (iii) from (i) and (ii). \end{proof}

\subsection{The structure of the set of mod-$p$ forms with biased distribution}

The following statement will be essential to our proofs of Theorem \ref{Almost positive correlations} and Theorem \ref{Overlap bounded below}. As we will show it follows from Proposition \ref{Equidistribution of k-tuples of linear forms} and information-theoretic techniques similar to those used in the proof of (i), Proposition \ref{Bounded number of irreducible sets}. We will write $U_p$ for the uniform distribution on $\F_p$. For $\ca$ a non-empty subset of $\{0,1\}^n$ and for $\phi: \F_p^n \rightarrow \F_p$ a linear form, we will write $\TV_{\ca}(\phi, U_p)$ for the total variation distance between the distribution of $\phi$ on $\ca$ and the distribution $U_p$.

\begin{proposition}\label{The set of mod p forms not about uniform on a dense subset of the cube is contained in boundedly many balls} Let $p$ be a prime, and let $\delta, \alpha>0$. There exist positive integers $B$ and $r$ depending on $\delta, \alpha$ (and $p$) only such that if $\ca \subset \{0,1\}^n$ has density at least $\delta$ inside $\{0,1\}^n$, then the family of mod-$p$ forms such that $\TV_{\ca}(\phi, U_p) \ge \a$ is contained in a union of $B$ balls of radius $r$ (for the support distance). \end{proposition}

\begin{proof} Let $l,r$ be positive integers which we will choose later. Assume that there exist mod-$p$ forms $\phi_1, \dots, \phi_l$ such that the linear combination \[a_1 \phi_1 + \dots + a_l \phi_l\] has support size at least $r$ for every $(a_1, \dots, a_r) \in \F_p^r \setminus \{0\}$, and such that $\TV_{\ca}(\phi_i, U_p) \ge \a$ for every $i \in [l]$. We write $\phi$ for the $l$-tuple $(\phi_1, \dots, \phi_l)$. By the contrapositive of Lemma \ref{Low conditional negentropy} there exists $c(\a)>0$ such that $\J_{\ca}(\phi_i) \ge c(\a)$ for every $i \in [l]$, so in particular \begin{equation} \J_{\ca}(\phi) \ge \sum_{i=1}^l \J_{\ca}(\phi_i) \ge c(\a) l. \label{First inequality on J}\end{equation} On the other hand \[\HH_{\ca}(\phi) = - \sum_{u \in \F_p^l} \P_{\ca}(\phi = u) \log \P_{\ca}(\phi = u),\] so using that \[\P_{\ca}(\phi = u) \le \delta^{-1} \P_{\{0,1\}^n}(\phi = u)\] for every $u \in \F_p^l$ we get \[\log \P_{\ca}(\phi = u) \le \log \P_{\{0,1\}^n}(\phi = u) + \log \delta^{-1}\] for every $u \in \F_p^l$, and hence the lower bound \[\HH_{\ca}(\phi) \ge - \sum_{u \in \F_p^l} \P_{\ca}(\phi = u) \log \P_{\{0,1\}^n}(\phi = u) - \log \delta^{-1}.\] Proposition \ref{Equidistribution of k-tuples of linear forms} shows that we can choose $r$ such that \[|\log \P_{\{0,1\}^n}(\phi = u) - \log p^{-l} | \le c(\a)l/2\] for every $u \in \F_p^k$ and hence \[\HH_{\ca}(\phi) \ge k \log p - c(\a)l/2 - \log \delta^{-1};\] in other words \[\J_{\ca}(\phi) \le c(\a)l/2 + \log \delta^{-1}.\] Combining this inequality with \eqref{First inequality on J} we obtain \[c(\a) l \le c(\a)l/2 + \log \delta^{-1},\] from which it follows that \[l \le 2 \log \delta^{-1} / c(\a).\] Since the linear span of the mod-$p$ forms $\phi_1, \dots, \phi_l$ contains at most $p^l$ mod-$p$ forms, the desired result is obtained with $B = p^l$.\end{proof}

\subsection{A sunflower-like lemma}

Finally, the last tool that we will use is the following sunflower-like lemma, which we state in two versions: one version which applies to sets and one more general ``metric entropy" version which applies to balls. For the purposes of proving Theorem \ref{Almost positive correlations} we shall only directly use the metric entropy version, but the set version will also be directly useful for us in the proof of Theorem \ref{Overlap bounded below}.

\begin{lemma}\label{Common element lemma}

We have the two following statements.
	
\begin{enumerate}

\item Set version: Let $C \ge 1$ be an integer, and let $\e>0$. Let $(A_i)_{i \in L}$ be a finite family of sets each with size at most $C$. Let us assume that for a proportion at least $\epsilon$ of the pairs $(i,j) \in L^2$ the intersection $A_i \cap A_j$ is not empty. Then there exists a subset $L_1$ of $L$ such that $|L_1| \ge (\epsilon/C) |L|$, and an element $x$ which belongs to all sets $A_i$ with $i \in L_1$.

\item Metric entropy version: Let $r \ge 0, C \ge 1$ be integers, and let $\e>0$. Let $(A_i)_{i \in L}$ be a finite family of sets each with size at most $C$. Let us assume that for a fraction at least $\epsilon$ of the pairs $(i,j) \in L^2$, there exist $x \in A_i$ and $y \in A_j$ such that \begin{equation} B(x, r) \cap B(y, r) \neq \emptyset. \label{Non-empty intersection of balls} \end{equation} Then there exist a subset $L_1$ of $L$ such that $|L_1| \ge (\epsilon/C)|L|$, an element $i \in L_1$ and an element $x \in A_i$ such that for each $j \in L_1$, there exists $y \in A_j$ such that \[B(y,r) \subset B(x, 3r).\]  \end{enumerate} \end{lemma}

\begin{proof} The set version is the special case $r=0$ of the metric entropy version, so it suffices to prove that latter version. By the assumption and an averaging argument there exists $i \in L$ such that for a proportion at least $\epsilon$ of the indices $j \in L$, there exist $x \in A_i$ and $y \in A_j$ satisfying \eqref{Non-empty intersection of balls}. Let $L'$ be the family of these indices $j$. Since $|A_i| \le C$, by the pigeonhole principle there exists $x \in A_i$ such that for a proportion at least $1/C$ of the indices $j \in L'$ there exists $y \in A_j$ satisfying \eqref{Non-empty intersection of balls}, so in particular $B(y,r) \subset B(x, 3r)$. We then take $L_1$ to be the set of these indices $j$. \end{proof}

As we will explain later in our proofs in Section \ref{Section: Proofs of the main theorems}, iterating Lemma \ref{Common element lemma} several times in its set version shows that we can find a dense subset $\Lambda$ of $[s]$ and a set $A$ with size at most $C$ such that $A_i \cap A_j \subset A$ for a proportion at least $1-\e$ of the pairs $(i,j) \in \Lambda^2$. Iterating the metric entropy version several times gives an analogous consequence: we can find a dense subset $\Lambda$ of $[s]$ and a set $A$ of at most $C$ elements such that the intersection $A_i \cap A_j$ is contained in the union of balls with radius $3r$ centered at the elements of $A$ for a proportion at least $1-\e$ of the pairs $(i,j) \in \Lambda^2$.

\section{Proofs of the main results}\label{Section: Proofs of the main theorems}

Whenever $f$ is a function defined on $\{0,1\}^n$, we write $\cod f$ for the codomain of $f$. 

\subsection{Obtaining close distributions}

We first set out to prove Proposition \ref{Close distributions for all functions depending on a bounded number of coordinates}. To do this, we first use irreducible sets to give a condition which entails close distributions. We prove it in slightly greater generality than is needed in our proof of Proposition \ref{Close distributions for all functions depending on a bounded number of coordinates} as we shall use the more general version later on.

\begin{proposition}\label{Establishing proximities in distribution} Let $k \ge 1$, $D_0 \ge 1$ be positive integers, and let $\eta>0$. Let $\mathcal{F}$ be a family of functions defined on $\{0,1\}^n$ each with finite codomain of size at most $D_0$. Let $\ca, \cb$ be non-empty subsets of $\{0,1\}^n$, and let $I_{\ca}$ and $I_{\cb}$ be the respective unions of all sets in $\mathcal{I}_{\le k}(\ca, \mathcal{F}, \eta)$ and in $\mathcal{I}_{\le k}(\cb, \mathcal{F}, \eta)$. Assume that $I_{\ca} = I_{\cb}$ and that $\TV_{\ca, \cb}(I_{\ca}) \le \eta$. Then \[\TV_{\ca, \cb}(T) \le 10 D_0^k k^{1/2} \eta^{1/4} = o_{\eta \rightarrow 0, k, D_0}(1)\] for every $T \subset \mathcal{F}$ with $|T| \le k$. 

\end{proposition}

\begin{proof} We write $I$ for the set $I_{\ca} = I_{\cb}$. Let $T \subset \mathcal{F}$ with $|T| \le k$, and let $x$ be a value in the codomain of $(f: f \in T)$. We can write \begin{equation}\P_{\ca}[T = x] = \P_{\ca}[I = x_{|I}] \P_{\ca}[T \setminus I = x_{|T \setminus I} | I = x_{|I}] \label{decomposition in the proof of proximities in distribution} \end{equation} and similarly for the probability $\P_{\cb}[T = x]$.
	
Assume that $\P_{\ca}[I = x_{|I}] \le \eta^{1/2}/2$ or $\P_{\cb}[I = x_{|I}] \le \eta^{1/2}/2$. Without loss of generality we can assume that the first bound holds. Then by the assumption on $\TV_{\ca, \cb}$ we have $\P_{\cb}[I = x_{|I}] \le \eta^{1/2}/2 + \eta$. The decomposition \eqref{decomposition in the proof of proximities in distribution} and its analogue for $\cb$ then show that $\P_{\ca}[T=x]$ and $\P_{\cb}[T=x]$ are both at most $\eta^{1/2}/2 + \eta$, so in particular \[|\P_{\ca}[T=x] - \P_{\cb}[T=x]| \le \eta^{1/2} + \eta.\]

Assume now instead that $\P_{\ca}[I = x_{|I}]$ and $\P_{\ca}[I = x_{|I|}]$ are both at least $\eta^{1/2}/2$. Then by \eqref{Markov inequality} we have \[\J_{\ca}(T \setminus I | I = x_{|I}) \le 2 \eta^{-1/2} \J_{\ca}(T \setminus I),\] so by Lemma \ref{Low conditional negentropy} we get \[\J_{\ca}(T \setminus I | I = x_{|I}) \le 4 \eta^{-1/2} k \eta,\] and then by Lemma \ref{Low negentropy implies approximate equidistribution} we obtain \[|\P_{\ca}[T \setminus I = x_{|T \setminus I} | I = x_{|I}] - 1/(|\cod(T \setminus I)|| \le ((32 k \eta^{1/2}/|\cod(T \setminus I)|)^{1/2}.\] By our assumption we have \[|\P_{\ca}[I = x_{|I}] - \P_{\cb}[I = x_{|I}]| \le \eta,\] so by the decomposition \eqref{decomposition in the proof of proximities in distribution} and its analogue for $\cb$ we obtain \[|\P_{\ca}[T = x] - \P_{\cb}[T = x]| \le 2 \eta + (32k \eta^{1/2}/|\cod(T \setminus I)|)^{1/2} \le 2 \eta + 8 k^{1/2} \eta^{1/4}. \] Summing over all $x \in \cod (f: f \in T)$ we therefore conclude \[\TV_{\ca,\cb}(T) \le D_0^k (2 \eta + 8 k^{1/2} \eta^{1/4}). \qedhere \] \end{proof}

We now deduce a statement which implies Proposition \ref{Close distributions for all functions depending on a bounded number of coordinates} as a special case.

\begin{proposition} \label{Close distributions for all forms with bounded support} Let $k,K,D_0 \ge 1$ be integers. Let $\mathcal{F}$ be a family of functions defined on $\{0,1\}^n$ each with finite codomain of size at most $D_0$. Let $\eta \in (0,1)$. Let $(\ca_i)_{i \in L}$ be a finite family of non-empty subsets of $\{0,1\}^n$. We assume \begin{equation} |\mathcal{I}_{\le k}(\ca_i, \mathcal{F}, \eta)| \le K \label{assumption that the irreducible families have bounded size} \end{equation} for each $i \in L$. Then there exists $N(k, K)$ such that if $|\cf| \ge N$, then there exists $L' \subset L$ with $|L'| \ge \frac{(\eta/Kk)^{D_0^{Kk}}}{2 \binom{|\cf|}{kK}} |L|$ satisfying \begin{equation} \Diam_{\ca_i: i \in L'} (g) \le 10 D_0^k k^{1/2} \eta^{1/4} \label{functions depending on few coordinates have close distributions} \end{equation} for every function $g$ of the type $F \circ (f_1, \dots, f_k)$ where $f_1, \dots, f_k \in \mathcal{F}$ and $F$ is some arbitrary function defined on $\cod (f_1) \times \dots \times \cod (f_k)$.

\end{proposition}

\begin{proof} For each $i \in L$ let $I_i = I_{\le k}(\ca_i, \mathcal{F}, \eta)$. By our assumption \eqref{assumption that the irreducible families have bounded size} and a union bound we obtain $|I_i| \le kK$ for each $i \in L$. There are therefore at most \[\sum_{0 \le t \le kK} \binom{|\cf|}{t} \le 2 \binom{|\cf|}{kK}\] possibilities for each set $I_i$, provided that $|\cf|$ is large enough. By the pigeonhole principle there exists a subset $L_1$ of $L$ of size at least $|L|/ (2\binom{|\cf|}{kK})$ and a set $I \subset \lbrack n \rbrack$ with $|I| \le kK$ such that $I_i = I$ for every $i \in L_1$. Applying the pigeonhole principle again, there exists a subset $L'$ of $L_1$ with size at least $(\eta/Kk)^{D_0^{Kk}}|L_1|$ and a probability distribution $\Delta$ on $\prod_{i \in I} \cod (f_i)$ such that $I_i = I$ for every $i \in L'$, and the distribution on $\ca_i$ of the map $(f: f \in I)$ is within total variation distance at most $\eta$ of $\Delta$. We conclude by Proposition \ref{Establishing proximities in distribution}. \end{proof}

Using Proposition \ref{Bounded number of irreducible sets} we obtain Proposition \ref{Close distributions for all functions depending on a bounded number of coordinates} as a special case of Proposition \ref{Close distributions for all forms with bounded support} by taking $\mathcal{F} = \mathcal{F}_0$, $D_0 = 2$ and $K = (k+1)! (\eta^{-1}\log (2\delta^{-1}))^k$.

\subsection{Obtaining almost positive correlations}

We next set out to prove Theorem \ref{Almost positive correlations}. If we now consider a family of superpolynomially many dense subsets $\ca_i \subset \{0,1\}^n$ rather than just one, and assume that the centres $\phi_1, \dots \phi_b$ with $b \le B$ of the balls involved in Proposition \ref{The set of mod p forms not about uniform on a dense subset of the cube is contained in boundedly many balls} are the same as $i$ varies, then applying Proposition \ref{Close distributions for all forms with bounded support} to the family $\mathcal{F}_0 \cup \{\phi_1, \dots \phi_b\}$ and to $k=b+r$ shows that there exist distinct $i,j$ such that every mod-$p$ form which is within support distance at most $r$ of one of the mod-$p$ forms $\phi_1, \dots \phi_b$ has close distributions on $\ca_i$ and $\ca_j$. As this is still the case for any other mod-$p$ form (since it is approximately uniformly distributed on both $\ca_i$ and $\ca_j$), we conclude (by taking $\a$ to be sufficiently small and then $n$ to be sufficiently large) that every mod-$p$ form has close distributions on $\ca_i$ and $\ca_j$.

However, in general the centres of the balls can vary from one index $i$ to another, which is why the proof we have just suggested will not work in general (and a conclusion as strong as that of Proposition \ref{Close distributions for all functions depending on a bounded number of coordinates} cannot hold for all mod-$p$ forms, as we have already discussed in Example \ref{No close distributions}). Instead we shall resort to Lemma \ref{Common element lemma} to obtain a situation which is more similar to the one that we just described.

The next proposition provides us with a structure that will allow us to conclude almost positive correlations.

\begin{proposition}\label{Set R for the proof of correlation} Let $p$ be a prime, let $m \ge 2$ be an integer, and let $\delta>0, \eta>0, \alpha>0, \epsilon>0$. Then there exist positive integers $r,B$ depending on $\delta, \alpha$ (and $p$) only and positive integers $A,A',C>0$ depending on $\delta, \eta, \a, \epsilon, m$ (and $p$) only, such that the following holds. For all integers $s \ge 1$ and all non-empty subsets $\ca_1, \dots, \ca_s$ of $\{0,1\}^n$ each with density at least $\delta$, there exist a subset $R \subset \lbrack s \rbrack$ with $|R| \ge s/(A'n^A)$, a set $\Phi$ of mod-$p$ forms with $|\Phi| \le B$, and a subset $I \subset \mathcal{F}_0 \cup \Phi$ with $|I| \le C$ such that the three following properties are satisfied.

\begin{enumerate}[(i)]

\item For a proportion at least $(1-\epsilon)$ of the $m$-tuples $(i_1, \dots, i_m) \in R^m$, if $\phi$ is a mod-$p$ form such that $\TV_{\ca}(\phi, U_p) \ge \a$ for at least two distinct indices $t \in \lbrack m \rbrack$, then $\phi$ is contained in some ball $B(\phi, 3r)$ for some linear combination $\phi$ of elements of $\Phi$.

\item We have $I_{\le 3r+B}(\ca_i, \mathcal{F}_0 \cup \Phi, \eta) = I$ for every $i \in R$.

\item The diameter $\Diam_{\ca_{i}: i \in R}(I)$ is at most $\eta$.

\end{enumerate}

\end{proposition}

\begin{proof} By Proposition \ref{The set of mod p forms not about uniform on a dense subset of the cube is contained in boundedly many balls} there exist positive integers $r$ and $B$ (depending on $\delta$ and $\a$) such that for each $i \in \lbrack s \rbrack$ we can associate with $\ca_i$ a family $\Phi_i$ of size at most $B$ of mod-$p$ forms such that every mod-$p$ form $\phi$ satisfying $\TV_{\ca_i}(\phi, U_p) \le \a$ is contained in the union of balls $\bigcup_{\psi \in \Phi_i} B(\psi, r)$. 

We then iterate the metric entropy version of Lemma \ref{Common element lemma} on the sets $\Phi_i$ as follows to obtain a set $\Phi$ of ``common" mod-$p$ forms. Let $\epsilon'$ be a quantity that will depend on $\epsilon$ and $n$ and which we shall fix later.

If a proportion at least $\epsilon'$ of the pairs $(i,j)$ of indices in $\lbrack s \rbrack$ are such that $\ca_i$ and $\ca_j$ satisfy \eqref{Non-empty intersection of balls} for some $x \in \Phi_i$ and $y \in \Phi_j$, then applying Lemma \ref{Common element lemma} we obtain a subset $L_1 \subset \lbrack s \rbrack$ with $|L_1| \ge (\epsilon'/B) s$ and a mod-$p$ form $\phi_1$ such that for every $j \in L_1$ there exists $\phi \in \Phi_j$ satisfying $B(\phi, r) \subset B(\phi_1, 3r)$. We then define $\Phi_j^1:= \Phi_j \setminus \{\phi\}$ for each $j \in L_1$. Otherwise, we stop the process.

More generally, at step $a$ of the process, if a proportion at least $\epsilon'$ of the pairs $(i,j)$ of elements of $L_{a-1}$ satisfies \eqref{Non-empty intersection of balls} for some $x \in \Phi_i^{a-1}$ and $y \in \Phi_j^{a-1}$, then applying Lemma \ref{Common element lemma} we obtain a subset $L_a \subset L_{a-1}$ with $|L_a| \ge (\epsilon'/B) |L_{a-1}|$ and a mod-$p$ form $\phi_a$ such that for every $j \in L_a$ there exists $\phi \in \Phi_j$ satisfying $B(\phi, r) \subset B(\phi_a, 3r)$. We then define $\Phi_j^a$ to be $\Phi_j^{a-1} \setminus \{\phi\}$ for each $j \in L_a$. Otherwise, we stop the process.

Let $a_{\max}$ be the number of steps that we take before the process stops. We necessarily have $a_{\max} \le B$, since taking $j \in L_{a_{\max}}$, we have $|\Phi^j| \le B$ and for each $a \in \lbrack a_{\max} \rbrack$ we have $|\Phi_j^a| \le |\Phi_j^{a-1}| - 1$ (where we write $\Phi_j^0$ for $\Phi_j$). We then let \[\Phi = \{\phi_1, \dots, \phi_{a_{\max}}\}\] and obtain a subset $L_{a_{\max}} \subset \lbrack s \rbrack$ with size at least $(\epsilon'/B)^B s$ such that for a proportion at least $1-\epsilon'$ of the pairs $(i,j) \in L_{a_{\max}}^2$ we have \[B(x,r) \cap B(y,r) = \emptyset\] for all $x \in \Phi_i \setminus \bigcup_{1 \le a \le a_{\max}} B(\phi_a, 3r)$ and $y \in \Phi_j \setminus \bigcup_{1 \le a \le a_{\max}} B(\phi_a, 3r)$.

We then assign to each $i \in L_{a_{\max}}$ the support \[I_{\le 3r+B}^{(i)}:= I_{\le 3r+B}(\ca_i, \mathcal{F}_0 \cup \Phi, \eta)\] of the irreducible sets with size at most $3r+B$, which by Proposition \ref{Bounded number of irreducible sets} applied with $f = (\phi_1, \dots, \phi_{a_{\max}})$ and the union bound has size at most some $C$ depending on $\delta, \eta, \a$ only.

There are therefore at most $\sum_{0 \le t \le C} \binom{n+B}{t}$ possibilities for $I_{3r+B}^{(i)}$. For $n$ large enough this is at most \[2 \binom{n+B}{C} \le 2(n+B)^C \le 4n^C.\] By the pigeonhole principle there therefore exists a subset $R_{\mathrm{pre}}$ of $L_{a_{\max}}$ with size at least $(4n^C)^{-1} |L_{a_{\max}}|$ and a subset $I \subset \mathcal{F}_0 \cup \Phi$ such that $I_{\le 3r+B}^{(i)} = I$ for all $i \in R_{\mathrm{pre}}$.

The family $I$, considered jointly as one function, has a codomain with size at most $p^C$. Applying the pigeonhole principle on the distribution of $I$ we obtain a subset $R \subset R_{\mathrm{pre}}$ with $|R| \ge (\eta/p^C)^{p^C} |R_{\mathrm{pre}}|$ such that $\max\limits_{(i,j) \in R^2} \TV_{\ca_i, \ca_j} I \le \eta$. Choosing \[\epsilon' = ((\eta/p^C)^{p^C}/(4n^C))^{2} \epsilon\] then ensures that (i) is satisfied in the case $m=2$, and \[\epsilon' = \binom{m}{2}^{-1} ((\eta/p^C)^{p^C}/ (4n^C))^{2} \epsilon\] suffices for general $m$ by a union bound. \end{proof}

We can now conclude the proof of Theorem \ref{Almost positive correlations}.

\begin{proof}[Proof of Theorem \ref{Almost positive correlations}] Let $\nu>0$ be fixed. We shall choose $\alpha$ depending on $\nu$, then $\eta$ depending on $\alpha$ and $\nu$, and then apply Proposition \ref{Set R for the proof of correlation}. Let $\Phi$ and $r$ be as in Proposition \ref{Set R for the proof of correlation}. We fix $(i_1, \dots, i_m)$ such that (i) from Proposition \ref{Set R for the proof of correlation} is satisfied, and consider a mod-$p$ form $\phi$.

Let us first assume that $\phi$ is outside the union of balls $\bigcup_{\psi \in \Phi} B(\psi, 3r)$. Then by (i) from Proposition \ref{Set R for the proof of correlation} we have $\TV(\phi, U_p) \le \a$ for all but at most one index $t \in \lbrack m \rbrack$. Choosing $\a$ sufficiently small (depending on $m,\nu$ only) then guarantees $\cor_{\ca_{i_1}, \dots, \ca_{i_m}} \phi \ge - \nu$.

If instead $\phi \in \bigcup_{\psi \in \Phi} B(\psi, 3r)$, then $\phi$ can be expressed as a linear combination (so in particular, as a function) of at most $3r+1$ functions in $\Phi \cup \mathcal{F}_0$. Using (ii) and (iii) from Proposition \ref{Set R for the proof of correlation} and applying Proposition \ref{Establishing proximities in distribution} we obtain $\TV_{\ca_{i_1}, \dots, \ca_{i_m}} \phi = o_{\eta \rightarrow 0,r}(1)$. Using the first implication in the hierarchy discussed after Definition \ref{Ways to measure proximity between two probability distributions}, we can choose $\eta$ small enough depending on $r$ (so indirectly on $\alpha$) which ensures $\cor_{\ca_{i_1}, \dots, \ca_{i_m}} \phi \ge 0$ and in particular $\cor_{\ca_{i_1}, \dots, \ca_{i_m}} \phi \ge - \nu$. \end{proof}

\subsection{Obtaining overlap bounded below}

We finally set out to prove Theorem \ref{Overlap bounded below}. We begin with the following lemma, which provides a condition that suffices to guarantee lower bounds on the overlaps of distributions of mod-$p$ linear forms on dense subsets of the cube.

\begin{proposition}\label{Obtaining an overlap bounded below}

Let $k \ge 1, m \ge 2$, $B \ge 1$, $r \ge 0$ be integers such that $B+r \le k$. Let $\eta > 0$.  Let $\Phi$ be a set of linear forms $\F_p^n \rightarrow \F_p$ with size $B$. Let $\ca_1, \dots, \ca_m$ be non-empty subsets of $\{0,1\}^n$, and let $I_{\ca_1}, \dots, I_{\ca_m}$ be the unions of all sets in \[\mathcal{I}_{\le k}(\ca_1, \mathcal{F}_0 \cup \Phi, \eta), \dots, \mathcal{I}_{\le k}(\ca_m, \mathcal{F}_0 \cup \Phi, \eta)\] respectively. Assume that there exist pairwise disjoint subsets $I_0, I_1, \dots, I_m$ of $\mathcal{F}_0 \cup \Phi$ such that $I_{\ca_i} \subset I_0 \cup I_i$ for each $i \in \lbrack m \rbrack$, and furthermore such that $\Diam_{\ca_1, \dots \ca_m}(I_0) \le \eta$. Then every linear form $\phi$ in the union of balls $\bigcup_{\psi \in \Phi} B(\psi, r)$ satisfies \[\omega_{\ca_1, \dots, \ca_m}(\phi) \ge 2^{-(m+1)(p-1)} - o_{\eta \rightarrow 0, k, m}(1). \]

\end{proposition}

\begin{proof}

By the assumption on $\phi$ we can in particular write \begin{equation} \phi = \psi + \sum_{z \in Z} a_z x_z \label{expression of phi in overlap proof} \end{equation} for some $\psi \in \Phi$, some subset $Z \subset \lbrack n \rbrack$ with size at most $r$, and some coefficients $a_z  \in \F$ for each $z \in Z$. Let $T$ be the set \[\{\psi\} \cup \{x_z: z \in Z\},\] which has size at most $k$.

We partition $T$ into the sets $T_0 = T \cap I_0$, $T_1 =  T \cap I_1$,\dots, $T_m =  T \cap I_m$, and $T_{\mathrm{out}} = T \setminus (\bigcup_{0 \le i \le m} T_i)$. Writing $\phi_0, \phi_i, \phi_{\mathrm{out}}$ for the respective contributions of $T_0, T_i, T_{\mathrm{out}}$ to the expression \eqref{expression of phi in overlap proof} of $\phi$, we obtain \begin{equation} \phi = \phi_0 + \sum_{i=1}^m \phi_i + \phi_{\mathrm{out}} \label{decomposition in obtaining overlap proof} \end{equation} as the resulting decomposition.

By hypothesis, $\TV_{\ca_i, \ca_1}(T_0) \le \eta$ for each $i \in \lbrack m \rbrack$, so since $\phi_0$ is a function of $T_0$, we have \begin{equation} \TV_{\ca_i, \ca_1}(\phi_0) \le \eta \label{Close probabilities for the centers} \end{equation} for each $i \in \lbrack m \rbrack$. Letting $U$ be the set of elements $u$ of $\F_p$ such that $\P_{\ca_1} [\phi_0=u] \ge \eta^{1/20} + \eta$, by \eqref{Close probabilities for the centers} we also have $\P_{\ca_i} [\phi_0=u] \ge \eta^{1/20}$ for every $i \in \lbrack m \rbrack$ and every $u \in U$.

Let $u \in U$ be fixed. We then choose $v(u) = (v_1, \dots, v_m) \in \F_p^m$ by selecting the $v_i$ successively for each $i=1,\dots,m$. At the $i$th iteration we choose $v_i$ such that \begin{equation} \P_{\ca_i}[\phi_i = v_i | \phi_0 = u, \phi_1 = v_1, \dots, \phi_{i-1} = v_{i-1}] \ge p^{-1} \label{Lower bound from pigeonhole for forms proof} \end{equation} (such a $v_i$ always exists provided that $\P_{\ca_i}[\phi_0 = u, \phi_1 = v_1, \dots, \phi_{i-1} = v_{i-1}] \neq 0$, which we guarantee by the previous iterations). We now show that \eqref{Lower bound from pigeonhole for forms proof} then allows us to deduce

\begin{equation} \P_{\ca_j}[\phi_i = v_i | \phi_0 = u, \phi_1 = v_1, \dots, \phi_{i-1} = v_{i-1}] \ge 2^{-(p-1)} - o_{\eta \rightarrow 0, k, m}(1)  \label{Lower bound from low negentropy for forms proof} \end{equation} for each $j \in \lbrack m \rbrack \setminus \lbrace i \rbrace$ and hence for each $j \in \lbrack m \rbrack$.

Let $j \in \lbrack m \rbrack \setminus \{i\}$ be fixed. Using the disjointness of the sets $I_0, I_1, \dots, I_m$ (of mod-$p$ forms) and Lemma \ref{Low conditional negentropy} we have \[ \mathbb{J}_{\ca_j}(T_i| T_0, T_1, \dots, T_{i-1}) \le 2k \eta \] so by estimates analogous to those in the proof of Proposition \ref{Establishing proximities in distribution}, we have

\begin{equation} \P_{\ca_{j}} [T_i = y_i | T_0 = x, T_1 = y_1, \dots, T_{i-1} = y_{i-1}] = |\cod(T_i)|^{-1} + o_{\eta \rightarrow 0, k}(1) \end{equation} for all $(x,y_1,\dots,y_{i-1}) \in \cod (T_0) \times \cod (T_1) \times \dots \times \cod(T_{i-1})$ satisfying \[\P_{\ca_j}[T_0=x, T_1 = y_1, \dots, T_{i-1} = y_{i-1}] \ge \eta^{1/2},\] and for all $y_i \in \cod (T_i)$. Because $\cod (T_0) \times \cod (T_1) \times \dots \times \cod(T_{i-1})$ has size at most $p^k$ and because by the previous choices $v_1, \dots v_{i-1}$ we have \[\P_{\ca_{j}}[\phi_0 = u, \phi_1 = v_1, \dots, \phi_{i-1} = v_{i-1}] \ge \eta^{1/20} 2^{-(i-1)(p-1)} - o_{\eta \rightarrow 0, k, m}(1),\] the law of total probability shows that \[\P_{\ca_j}[T_i = y_i | \phi_0 = u, \phi_1 = v_1, \dots, \phi_{i-1} = v_{i-1}] = |\cod(T_i)|^{-1} + o_{\eta \rightarrow 0, k, m}(1)\] for each $y_i \in \cod (T_i)$. Since $\cod (T_i)$ has size at most $p^k$, the distribution of $T_i$ on $\ca_j$ conditionally on $\phi_0 = u, \phi_1 = v_1, \dots, \phi_{i-1} = v_{i-1}$ is within total variation distance $o_{\eta \rightarrow 0, k, m}(1)$ of the uniform distribution on $\cod (T_i)$.

Writing $\phi_i = \sum_{f \in T_i} a_f f$ as a linear combination of the elements of $T_i$, we can therefore approximate \[\P_{\ca_j}[\phi_i = v_i | \phi_0 = u, \phi_1 = v_1, \dots, \phi_{i-1} = v_{i-1}]\] within $o_{\eta \rightarrow 0, k, m}(1)$ by the probability $L$ that the linear combination $\sum_{f \in T_i} a_f X_f$ of independent random variables $X_f$ takes the value $v_i$, where $X_f$ is uniformly distributed on $\{0,1\}$ if $f$ is a coordinate (and hence has $\{0,1\}$ as its codomain), and $X_f$ is uniformly distributed on $\F_p$ if $f$ is a non-coordinate mod-$p$ form (and hence has $\F_p$ as its codomain).

We now show that $L \ge 2^{-(p-1)}$. If one of the functions $f \in T_i$ is the mod-$p$ form $\psi$, and furthermore $\psi$ is not the zero mod-$p$ form and $a_{\psi}$ is non-zero, then $L = 1/p$. Otherwise, all functions $f \in T_i$ such that $a_f f \neq 0$ are coordinates, so the distribution of $\sum_{f \in T_i} a_f X_f$ is the same as the distribution of $\phi_i$ on $\{0,1\}^n$; since by \eqref{Lower bound from pigeonhole for forms proof} the probability $\P_{\{0,1\}^n}[\phi_i=v_i]$ is non-zero, we have $L \neq 0$, and therefore $L \ge 2^{-(p-1)}$ by Proposition \ref{Lower bound on non-zero probabilities}. The estimate \eqref{Lower bound from low negentropy for forms proof} follows, which completes the inductive step.

Moreover, once $u$ and then $v(u)$ are fixed, we can similarly find $w(u) \in \F_p$ satisfying \begin{equation} \P_{\ca_{j}} [\phi_{\mathrm{out}} = w(u) | \phi_0 = u, \phi_1 = v(u)_1, \dots, \phi_m = v(u)_m] \ge 2^{-(p-1)} - o_{\eta \rightarrow 0, k, m}(1) \label{Lower bound for z for forms proof} \end{equation} for each $j \in \lbrack m \rbrack$. By the decomposition \eqref{decomposition in obtaining overlap proof} and the definition of the overlap we have the lower bound \[ \omega_{\ca_1, \dots, \ca_m}(T) \ge \sum_{u \in U} \min_{1 \le i \le m} \P_{\ca_i}[\phi_0 = u, \phi_1=v(u)_1, \dots, \phi_m = v(u)_m, \phi_{\mathrm{out}} = w(u)]. \] From the inequalities \eqref{Close probabilities for the centers}, \eqref{Lower bound from pigeonhole for forms proof}, \eqref{Lower bound from low negentropy for forms proof}, \eqref{Lower bound for z for forms proof} and the law of total probability, we obtain

\[\omega_{\ca_1, \dots, \ca_m}(T) \ge \sum_{u \in U} 2^{-(m+1)(p-1)} \P_{\ca_1}[\phi_0 = u] (1-o_{\eta \rightarrow 0, k, m}(1)).\] Since it follows from the definition of $U$ that \[\P_{\ca_1}[\phi_0 \in U] = 1 - o_{\eta \rightarrow 0, k}(1),\] we conclude that \[\omega_{\ca_1, \dots, \ca_m}(T) \ge 2^{-(m+1)(p-1)} - o_{\eta \rightarrow 0, k, m}(1). \qedhere \] \end{proof}

We next formulate an analogue of Proposition \ref{Set R for the proof of correlation}, which will provide us with a structure that will allow us to obtain a lower bound on the overlap of distributions by applying Proposition \ref{Obtaining an overlap bounded below}.

\begin{proposition}\label{Set Q for the proof of overlap} Let $p$ be a prime, let $m \ge 2$ be an integer, and let $\delta>0, \eta>0, \alpha>0, \epsilon>0$. Then there exist positive integers $r,B$ depending on $\delta, \alpha$ only and positive integers $A,C>0$ depending on $\delta, \eta, \a, \epsilon, m$ only, such that the following holds. For all integers $s \ge 1$ and all non-empty subsets $\ca_1, \dots, \ca_s$ of $\{0,1\}^n$, each with density at least $\delta$, there exist a subset $Q \subset \lbrack s \rbrack$ with $|Q| \ge As$, a set $\Phi$ of linear forms $\F_p^n \rightarrow \F_p$ with $|\Phi| \le B$, and a subset $I \subset \mathcal{F}_0 \cup \Phi$ with $|I| \le C$ such that the three following properties are satisfied.

\begin{enumerate}[(i)]

\item For a proportion at least $(1-\epsilon)$ of the $m$-tuples $(i_1, \dots, i_m) \in Q^m$, if $\phi: \F_p^n \rightarrow \F_p$ is a linear form such that $\TV_{\ca}(\phi, U_p) \ge \a$ for at least two distinct indices $t \in \lbrack m \rbrack$, then $\phi$ is contained in some ball $B(\phi, 3r)$ for some linear combination $\phi$ of elements of $\Phi$.

\item For a proportion at least $(1-\epsilon)$ of the $m$-tuples $(i_1, \dots, i_m) \in Q^m$, if $f$ is an element of $\mathcal{F}_0 \cup \Phi$ which belongs to $I_{\le 3r+B}(\ca_{i_t}, \mathcal{F}_0 \cup \Phi, \eta)$ for at least two distinct indices $t \in \lbrack m \rbrack$, then $f \in I$.

\item The diameter $\Diam_{\ca_{i}: i \in Q}(I)$ is at most $\eta$.

\end{enumerate}

\end{proposition}

\begin{proof}

As in the proof of Proposition \ref{Set R for the proof of correlation}, we begin by introducing a quantity $\epsilon'$ which we shall fix later and which depends on $\epsilon$. (However, it will be independent of $n$ this time.) We go through the proof of Proposition \ref{Set R for the proof of correlation} up to and including the point where $I_{\le 3r+B}^{(i)}$ is defined for each $i \in L_{a_{\max}}$. Then, rather than applying the pigeonhole principle to the sets $I_{\le 3r+B}^{(i)}$, we iteratively apply the set version of Lemma \ref{Common element lemma} to the sets $I_{\le 3r+B}^{(i)}$, just as we did to the sets $\ca_i$ with the metric entropy version at the start of the proof, this time with parameter some $\epsilon''$ which we shall also fix later. We obtain a subset $I$ of $\mathcal{F}_0 \cup \Phi$ and a subset $Q_{\mathrm{pre}}$ of $L_{a_{\max}}$ with size at least $(\epsilon''/C)^C |L_{a_{\max}}|$, such that for a proportion at least $1 - \epsilon''$ of the pairs $(i,j) \in Q_{\mathrm{pre}}^2$ we have \[I_{\le 3r+B}^{(i)} \cap I_{\le 3r+B}^{(j)} = I.\]

We then resume the argument as in the proof of Proposition \ref{Set R for the proof of correlation}, obtaining the set $Q$ from $Q_{\mathrm{pre}}$ in the same manner that $R$ was obtained from $R_{\mathrm{pre}}$. Taking successively \begin{align*} \epsilon'' & = \binom{m}{2}^{-1} (\eta/p^C)^{p^C})^2 \epsilon \\ \epsilon' &= \binom{m}{2}^{-1} ((\epsilon''/C)^C)^2 (\eta/p^C)^{p^C})^2 \epsilon \end{align*} then guarantees that (ii) and (i) are satisfied respectively. \end{proof}

We are now ready to deduce Theorem \ref{Overlap bounded below}.

\begin{proof}[Proof of Theorem \ref{Overlap bounded below}]

Let $\eta>0$ be fixed. We shall choose $\a$ in a manner that need not depend on $\nu$ (although our argument would still work if it did), then choose $\eta$ depending on $\nu$, and then apply Proposition \ref{Set Q for the proof of overlap}. We let $\Phi$ and $r$ be as in Proposition \ref{Set Q for the proof of overlap}. We fix $(i_1, \dots, i_m)$ such that (i) and (ii) from Proposition \ref{Set Q for the proof of overlap} are satisfied, and consider a mod-$p$ form $\phi$.

Let us first assume that $\phi$ is outside the union of balls $\bigcup_{\psi \in \Phi} B(\psi, 3r)$. Then by (i) from Proposition \ref{Set Q for the proof of overlap} we have $\TV_{\ca_{i_t}}(\phi, U_p) \le \a$ for all but at most one index $t \in \lbrack m \rbrack$ and hence \[\omega_{\ca_{i_1}, \dots, \ca_{i_m}} \phi \ge p^{-1} - o_{\a  \rightarrow 0}(1).\] We can choose $\a$ sufficiently small (in a manner that does not depend on $\nu$) such that this entails \[\omega_{\ca_{i_1}, \dots, \ca_{i_m}} \phi \ge 2^{-(m+1)(p-1)} - \nu.\]

If instead $\phi \in \bigcup_{\psi \in \Phi} B(\psi, 3r)$, then (ii) and (iii) from Proposition \ref{Set Q for the proof of overlap} provide us with the assumptions of Proposition \ref{Obtaining an overlap bounded below}, from which we obtain \[\omega_{\ca_{i_1}, \dots, \ca_{i_m}} \phi \ge 2^{-(m+1)(p-1)} - o_{\eta \rightarrow 0, \d, \a, m}(1).\] Choosing $\eta$ sufficiently small (depending on $\d,\a,m$) we conclude \[\omega_{\ca_{i_1}, \dots, \ca_{i_m}} \phi \ge 2^{-(m+1)(p-1)} - \nu\] as desired. \end{proof}

\end{document}